\documentclass[12pt]{article}

\usepackage[utf8]{inputenc}

\usepackage{geometry} 
\usepackage{amsmath}
\usepackage{amsthm}
\usepackage{amssymb}
\usepackage{graphicx}

\usepackage{pstricks}

\theoremstyle{plain}
\newtheorem{theorem}{Theorem}[section]
\newtheorem{lemma}{Lemma}
\newtheorem{corollary}[theorem]{Corollary}
\newtheorem{maintheorem}{Theorem}

\newcommand{\bd}{\mathrm{bd}}

    \makeatletter
    \renewenvironment{proof}[1][\proofname]{%
      \par\pushQED{\qed}\normalfont%
      \topsep6\p@\@plus6\p@\relax
      \trivlist\item[\hskip\labelsep\bfseries#1\@addpunct{.}]%
      \ignorespaces
    }{%
      \popQED\endtrivlist\@endpefalse
    }
    \makeatother

\title{Convex bodies with many elliptic sections}

\author{Isaac Arelio \and Luis Montejano}
\date{} 

\begin{document}
\maketitle

\abstract{We show in this paper that two normal elliptic sections through every point of the boundary of a smooth convex body essentially characterize an ellipsoid and furthermore, that four different pairwise non-tangent elliptic sections through every point of the $C^2$-differentiable boundary of a convex body also essentially characterize an ellipsoid.}

\section{Introduction}

Bianchi and Gruber  \cite{gruber} proved that the boundary of a convex body $K\subset\mathbb{R}^n$  is an ellipsoid if for every direction, continuously, we can choose a hyperplane which intersects $\bd K$ in an ellipsoid. The proof of this result leads to the following characterization of the ellipsoid: 
Let $K\subset\mathbb{R}^3$ be a convex body and let $\alpha>0$.  If for every  support line of $K$ there is a plane $H$ containing it whose intersection with $\bd K$ is an ellipse of area at least $\alpha$, then $\bd K$ is an ellipsoid.

This characterization requires that for every $p\in \bd K$ and every direction in the support plane of $K$ at $p$ there is a section of $\bd K$ in that direction, containing $p$,  that is an ellipse. The aim of this work is to give a characterization of the ellipsoid where for every boundary point only a finite number of ellipses containing that point are required.

The sphere was characterized in this manner by Miyaoka and Takeuchi \cite{miyaoka,takeuchi} as the unique compact, simply connected $C^\infty$ surface that satisfies one of the following properties: i) contains three circles through each point; ii) contains two transversal circles through each point; or iii) contains one circle  inside a normal plane. In our paper, we shall show that  two normal elliptic sections through every point of the boundary of a smooth convex body essentially characterize an ellipsoid and furthermore, four different pairwise nontangent, elliptic sections through every point of the $C^2$-differentiable boundary of a convex body also characterize an ellipsoid.

The set $$C = \{(x , y , z )\in\mathbb{R}^3 |x^2 + y^2 + z^2 + \tfrac{5}{4} xyz\leq1,\newline max\{|x|,|y|,|z|\}\leq1\}$$ \cite[Example 2]{alonso} is a convex body whose boundary is not an ellipsoid. The planes parallel to $x = 0$, $y = 0$ and $z = 0$ intersect $\bd C$ in ellipses. Thus there are two ellipses for every point $p$ in the boundary of $C$ and there are three ellipses for every point $p$ in the boundary of $C$ except at six points. 

Before giving the precise statement of the main theorems, we introduce the concepts and results we will use.

\section{Definitions and auxiliary results}

Let us consider a continuous function $\Psi \colon \bd K\rightarrow\mathbb{S}^2$. For every $p\in \bd K$, let $L_p$ be the line through $p$ in the direction $\Psi(p)$. We say that $\Psi$ is an \emph{outward} function if for every $p\in \bd K$,
\begin{itemize}
\item the line $L_p$ is not tangent to $\bd K$, and $\{p+t\Psi(p) | t>0\}$ is not in $K$; and

\item there is a point $q\in \bd K\setminus L_p$ such that the line $L_q$ intersects the line $L_p$ at a point of the interior of $K$.
\end{itemize}

For example, the function $\eta \colon \bd K\rightarrow\mathbb{S}^2$ such that for every $p\in \bd K$, $\eta(p)$ is the normal unit vector to $\bd K$ at $p$ is an outward function. To see this, let $O$ be the midpoint of the interval $L_p\cap K$ and consider the farthest and the nearest point of $\bd K$ to the point $O\in \mathrm{int} K$. One of them, call it $q$, is not in $L_p$, otherwise $K$ would be a solid sphere, in which case the assertion is trivially true; hence clearly $O\in L_q$.

Let $K$ be a convex body and let $H_1$ and $H_2$ be two planes. In the following it will be useful to have a criterion to know when two sections $H_1\cap \bd K$ and $H_2\cap \bd K$ intersect in exactly two points. This holds exactly when the line $L = H_1\cap H_2$ intersects the interior of $K$.

We say that a collection of lines $\mathcal{L}$ in $\mathbb{R}^{n+1}$ is a \emph{system of lines} if for every direction $u\in\mathbb{S}^n$ there is a unique line $L_u$ in $\mathcal{L}$ parallel to $u$.

Given a system of lines $\mathcal{L}$ we define the function $\delta \colon \mathbb{S}^n\rightarrow\mathbb{R}^{n+1}$ which assigns to every direction $u\in\mathbb{S}^n$ the point $\delta(u)\in L_u$ which traverses the distance from the origin to the line $L_u$. When $\delta$ is continuous we say that $\mathcal{L}$ is a \emph{continuous system of lines}. Finally, the set of intersections between any two different lines of the system will be the \emph{center} of $\mathcal{L}$.

A continuous system of lines has a certain property that will be useful for our purpose which  is stated in the following lemma. 

\begin{lemma} \label{lem1} Let $\mathcal{L}$ be a continuous system of lines in $\mathbb{R}^{n+1}$. For every direction $u\in\mathbb{S}^n$, there exists $v\in\mathbb{S}^n$, $v\neq u$, such that the lines $L_u$ and $L_v$ have a common point. \end{lemma}

\begin{proof} Let $H$ be the plane through the origin that is orthogonal to $L_u$. Without loss of generality assume that $H$ is the plane $z = 0$ and $L_u$ contains the origin. Define $\mathcal{L}_H$ as the set of orthogonal projections of all lines in $\mathcal{L}$ parallel to $H$ onto $H$; $\mathcal{L}_H$ is then a system of lines in $H$.

By \cite[Proposition 3]{stein}, there is a line in $\mathcal{L}_H$ passing through the origin. This means that there is a direction $v$ in $H$ with the property that the line $L_v\in\mathcal{L}$ intersects $L_u$. \end{proof}


Suppose that there is a continuous system of lines $\mathcal{L}$ such that the center of $\mathcal{L}$ is contained in the interior of $K$. Note  that every line of $\mathcal{L}$ thus intersects the interior of $K$. For every $p\in \bd K$, let $L_p$ be the unique line of $\mathcal{L}$ through $p$. Let us define the continuous function $\Psi \colon \bd K\rightarrow\mathbb{S}^2$ in such a way that $\Psi(p)$ is the unique unit vector parallel to $L_p$ with the property that $\{p + t \Psi(p) | t > 0\}$ is not in $K$. By Lemma \ref{lem1}, $\Psi \colon \bd K\rightarrow\mathbb{S}^2$ is an outward function. If $K\subset\mathbb{R}^3$ is a strictly convex body, then the system of diametral lines of $K$ is a continuous system of lines whose center is contained in the interior of $K$. 

The following lemma will be of use to us in Section~4.

\begin{lemma} \label{lem2} Let $M_1$ and $M_2$ be two surfaces tangent at $p\in M_1\cap M_2$. If the normal sectional curvatures of $M_1$ and $M_2$ coincide in three different directions, then the normal sectional curvatures of $M_1$ and $M_2$ coincide in every direction. \end{lemma}
\begin{proof} Suppose the Euler curvature formula for the first surface $M_1$ is given by 
\[
\kappa_1\cos^2(\theta) +\kappa_2\sin^2(\theta),
\]
\noindent and for the second surface $M_2$ by
\[
\kappa^{\prime}_1\cos^2(\theta -t_0) +\kappa^{\prime}_2\sin^2(\theta-t_0).
\]
Suppose that the difference between these two expressions  has three zeros. After simplifying the difference, replacing $\cos^2(x)$ and $\sin^2(x)$ by their corresponding expressions in $cos(2x)$ and $sin(2x)$, we obtain an expression of the form $A+B\cos(2\theta)+C\sin(2\theta)$, where $A$, $B$ and $C$ depend only on the principal curvatures and the angle $t_0$. Since $\theta$ lies in the unit circle, where the number of zeros is bounded by the number of critical points, we may assume that the expression $C\cos(2\theta)-B\sin(2\theta)$ has at least three zeros. This implies that  $A+B\cos(2\theta)+C\sin(2\theta)=0$ and hence that the the normal sectional curvatures of $M_1$ and $M_2$ coincide in every direction. \end{proof}

\section{Two elliptic sections through a point}

\begin{theorem} \label{thm:ellipsoid} Let $K\subset\mathbb{R}^3$ be a convex body and $\alpha$ a given positive number. Suppose that there is an outward function $\Psi \colon \bd K\rightarrow \mathbb{S}^2$ such that for every $p\in \bd K$\nolinebreak  there are planes $H_1, H_2$ determining an angle at least $\alpha$, where $H_1\cap H_2$ is the line $L_p$ through $p$ in the direction $\Psi(p)$ and $K\cap H_i$ is an elliptic section, for $i=1,2$.  Then $\bd K$ is an ellipsoid. \end{theorem}

\begin{proof}
Let $p\in \bd K$. By hypothesis there are two planes $H_1$ and $H_2$ such that $L_p=H_1\cap H_2$ and $E_i=\bd  K\cap H_i$ is an ellipse, for $i=1,2$. Furthermore, there is a point $q\in \bd K$ such that the line $L_q$ intersects $L_p$ at a point in the interior of $K$. Hence at least one of the elliptic sections through $L_q$ is different from $E_1$ and $E_2$; call this section $E_3$. We have that $E_3$ has two points in common with $E_i$, for $i=1,2$, because $L_p\cap L_q$ belongs to the interior of both sections.

\begin{center}
\includegraphics[width=3in]{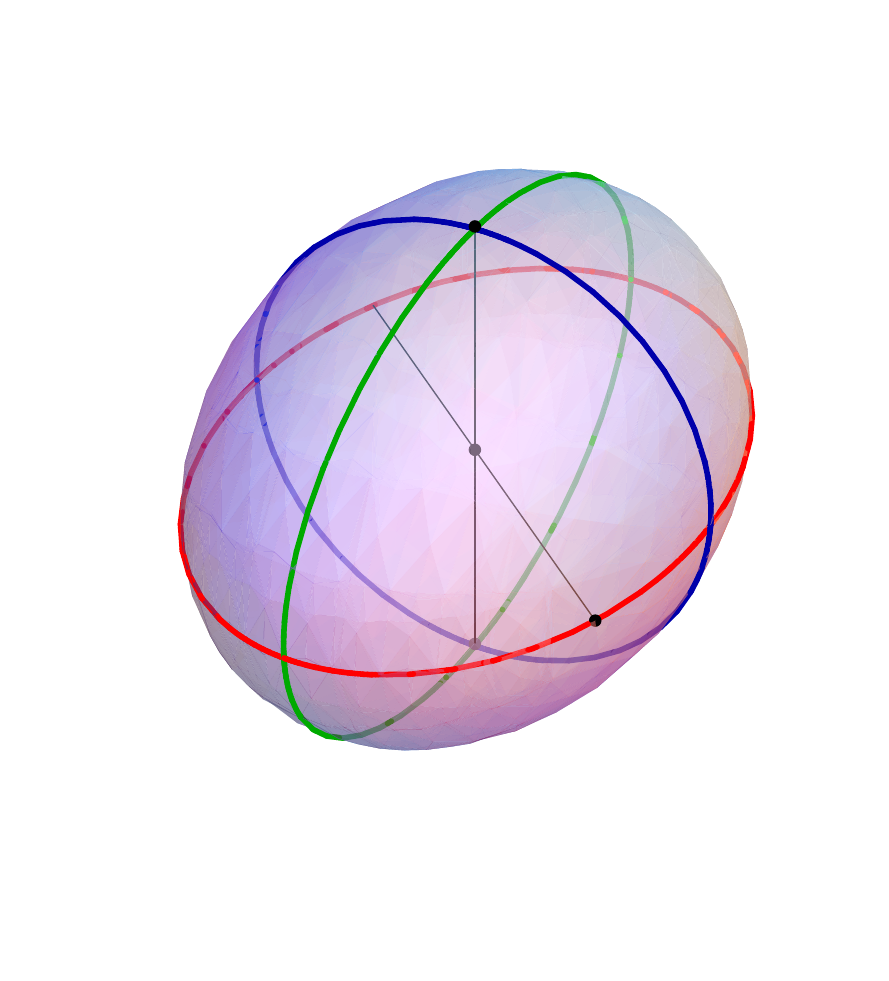}\rput(-5.5,6.2){$E_1$}\rput(-2.2,7){$E_2$}\rput(-.85,5){$E_3$}\rput(-2.5,3.3){$q$}\rput(-3.5,6.7){$p$}
\end{center}

We will now show that there is a quadric surface which contains $E_i$, for $i=1,2,3$. Let  \[E_1\cap E_2=\{p,p'\}\] and for $i=1,2$, \[E_i\cap E_3=\{p_{i3},{p'_{i3}}\}.\] We can choose $p_i\in \bd E_i$ distinct from $p$, $p'$, $p_{i3}$ and $p'_{i3}$ for $i=1,2$. Then the points $p,p',p_{13},p'_{13},p_{23},p'_{23},p_1,p_2$ and $p_3$ uniquely determine a quadric surface $Q$. Furthermore $E_i$ has five points in common with $Q$. This implies that $E_i\subset Q$, for $i=1,2,3$.


Now we will verify that there is an open neighborhood $\mathcal{N}$ of $p$ such that $\bd K\cap\mathcal{N}$ is contained in $Q$. Suppose that there is no such neighborhood. Then there is a sequence $\{q_n\}_{n\in\mathbb{N}}$ in $\bd K\setminus Q$ such that $\lim\limits_{n\rightarrow\infty}q_n=p$, and moreover, the lines $L_{q_n}$ converge to the line $L_p$. Our strategy is to prove that if $n$ is sufficiently big, then one of the elliptic sections of $K$ through $L_{q_n}$ intersects each elliptic section $E_i$ at two points. It is clear that if $n$ is sufficiently big, both elliptic sections of $K$ through $L_{q_n}$ intersect $E_3$ at two points, because $L_{q_n}$ intersects the interior of the ellipse $E_3$. The same holds if $L_{q_n}$ intersects the interior of the ellipse $E_i$. Suppose then that $L_{q_n}$ does not intersect the interior of the ellipse $E_i$, and let $N_1$ be the unit vector normal to the plane determined by $L_{q_n}$ and $p_{i3}$ in the direction of the semiplane not containing ${p_{i3}}'$, and let $N_2$ be the unit vector normal to the plane determined by $L_n$ and ${p_{i3}}'$ in the direction of the semiplane not containing ${p_{i3}}$. Let $\alpha_n$ be the angle determined by  $N_1$ and $N_2$.

\begin{center}
\includegraphics[width=2.65in]{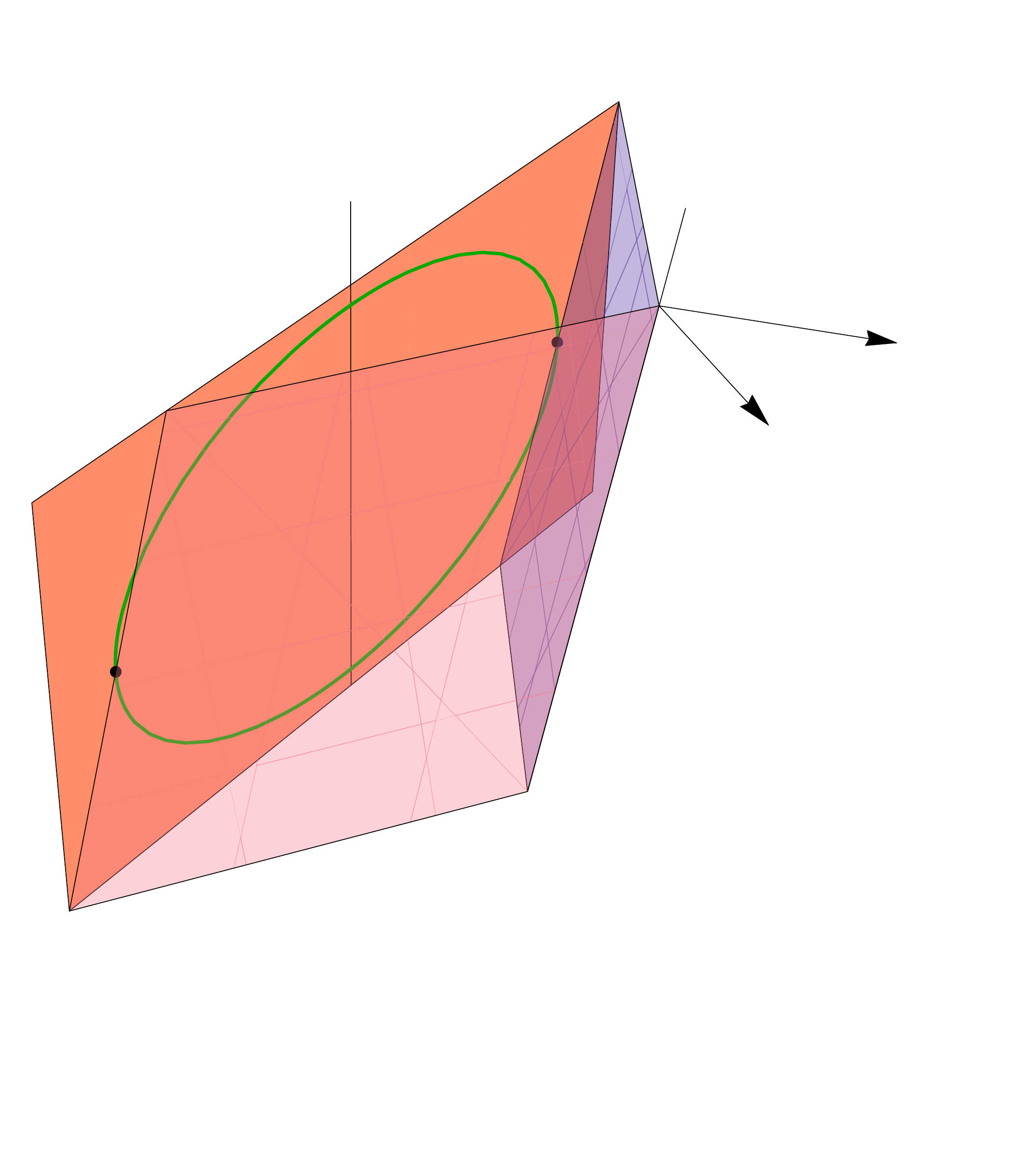}\rput(-1.75,5.4){$\alpha_n$}\rput(-3.35,5.5){$p'_{i3}$}\rput(-5.6,3.5){$p_{i3}$}\rput(-.6,5.4){$N_1$}\rput(-2,6.75){$L_{q_n}$}\rput(-4.3,6.75){$L_p$}\rput(-1.5,4.75){$N_2$}\rput(-5,4.75){$E_i$}
\end{center}

The convergence of $\{L_{q_n}\}_{n\in\mathbb{N}}$ to $L_p$ implies that $\lim\limits_{n\rightarrow\infty}\alpha_n=0$. Then there is $k_i\in\mathbb{N}$ such that for every $n\geq k_i$, we have  $\alpha_n <\alpha$. This means that at least one of the normal elliptic sections through $L_{q_n}$ intersects the relative interior of the segment between $p_{13}$ and $p'_{13}$ and therefore has two points in common with $E_i$.

Thus there is $k_0\in\mathbb{N}$ such that for $n > k_0$ there is a plane $H_n$ through $L_{q_n}$ which determines an elliptic section $F_n=H_n\cap \bd K$ and $F_n$ has $6$ points in common with $Q$. This shows that $F_n$ coincides with the conic $H_n \cap Q$, and hence that $q_n\in Q$. This proves that there is an open neighborhood $\mathcal{N}$ of $p$ such that $\mathcal{N}\cap \bd K\subset Q$.

We conclude that there is a finite open cover of $\bd K$ in which every element coincides with a quadric; by the connectedness of $\bd K$ and the fact that two quadrics that coincide in a relative open set of $\bd K$ must be the same quadric, $\bd K$ is thus contained in a quadric. Therefore $\bd K$ is an ellipsoid. \end{proof}

\begin{theorem} Let $K\subset\mathbb{R}^3$ be a convex body and $\alpha$ a given positive number. Suppose that there is a continuous system of lines $\mathcal{L}$ such that the center of $\mathcal{L}$ is contained in the interior of $K$ and through every $L$ in $\mathcal{L}$ there are planes $H_1$, $H_2$ determining an angle at least $\alpha$ such that  $K\cap H_i$ is an elliptic section, for $i = 1,2$. Then $\bd K$ is an ellipsoid. Moreover, if for every $L$ in $\mathcal{L}$ one of the elliptic sections is a circle, then $\bd K$ is a sphere. \end{theorem}

\begin{proof} Note first that every line of $\mathcal{L}$ intersects the interior of $K$. For every $p\in \bd K$, let $L_p$ be the unique line of $\mathcal{L}$ through $p$ and let us define the continuous function $\Psi \colon \bd K\rightarrow \mathbb{S}^2$ in such a way that $\Psi(p)$ is the unique unit  vector parallel to $L_p$ with the property that $\{p + t\Psi(p) | t > 0\}$ is not in $K$. By  Lemma \ref{lem1}, $\Psi \colon \bd K \rightarrow\mathbb{S}^2$ is an outward function and by Theorem \ref{thm:ellipsoid}, $\bd K$ is an ellipsoid.

Let $L_0$ in $\mathcal{L}$ be the line parallel to the diameter of $K$. Since  there is a circular section of $K$ through $L_0$ parallel to the diameter, and since every section of the ellipsoid $\bd K$ parallel to this circular section is also circular, then there is a circular section of the ellipsoid $\bd K$ through the diameter. This implies that the three axes of the ellipsoid $\bd K$ have the same length and therefore that $\bd K$ is a sphere. \end{proof}

Let $K\subset\mathbb{R}^3$ be a smooth, strictly convex body and let $p\in \bd K$. Suppose that $H$ is a plane through the unit normal vector of $K$ at $p$. Then we say that the section $H\cap K$ is a \emph{normal section} of $K$ at $p$. If $H$ is a plane containing the diametral line of $K$ through $p$, then we say that the section $H\cap K$ is a \emph{diametral section} of $K$ at $p$.

\begin{theorem} \label{thm:diametral} Let $K\subset\mathbb{R}^3$ be a smooth, strictly convex body and $\alpha$ a given positive number. Suppose that through every $p\in \bd K$ there are two elliptic normal (respectively diametral) sections determining an angle at least $\alpha$. Then $\bd K$ is an ellipsoid. \end{theorem}

Motivated by the fact that for every diametral line of an ellipsoid there are two sections of the same area, we have the following result.

\begin{corollary} Let $K\subset\mathbb{R}^3$ be a smooth, strictly convex body and $\alpha$ a given positive number. Suppose that through every diametral line there are three elliptic sections of the same area determining an angle at least $\alpha$. Then $\bd K$ is a sphere. \end{corollary}

\begin{proof} By Theorem \ref{thm:diametral}, $\bd K$ is an ellipsoid. Note now that the hypothesis implies that the section through the center of $K$ orthogonal to one of the axes is a circle. This implies that the three axes of the ellipsoid $\bd K$ have the same length. \end{proof}

\section{Four elliptic sections through a point}

\begin{maintheorem} Let $K\subset\mathbb{R}^3$ be a convex body with a $C^2$-differentiable boundary and let $\alpha> 0$. Suppose that through every $p\in \bd K$ there are four planes $H_{p_1}$, $H_{p_2}$, $H_{p_3}$, $H_{p_4}$ satisfying:
\begin{itemize}
\item $H_{p_j}\cap \bd K$ is an ellipse $E_{p_j}$ of area greater than $\alpha > 0$, $j=1,2,3,4$, 
\item for $1\leq i<j\leq 4$, the ellipses $E_{p_i}$ and $E_{p_j}$ are not tangent.
\end{itemize}
Then $\bd K$ is an ellipsoid. \end{maintheorem}

\begin{proof}
We shall prove that locally, $\bd K$ is a quadric.  Let $p\in \bd K$. Then the line $L_{ij}=H_{p_i}\cap H_{p_j}$ intersects the interior of $K$, because  the ellipses $E_{p_i}$ and $E_{p_j}$ are not tangent. For $1\leq i < j \leq 4$, let $\{p, p_{ij}\}=E_{p_i}\cap E_{p_j} = L_{ij}\cap \bd K$. We shall show that there is a quadric $Q$ that contains the four ellipses $E_{p_1}$, $E_{p_2}$, $E_{p_3}$, $E_{p_4}$ and from that, we shall prove that $Q$ coincides with $\bd K$ in a neighborhood of $p$.

In order to prove that $E_{p_1}$, $E_{p_2}$, $E_{p_3}$, $E_{p_4}$ are contained in a quadric $Q$, we follow the spirit of Gruber and Bianchi in \cite{gruber}.  If $p\in \bd K$, let $H_p$ be the support plane of $K$ at $p$ and let $L_{p_i}=H_{p_i}\cap H_p$, $i=1,\dots ,4$.

\medskip\noindent\textbf{Case a).}  No three of the planes $E_{p_1}$, $E_{p_2}$, $E_{p_3}$, $E_{p_4}$ share a line.  Then we have six distinct points $p_{ij}\in \bd K$. By elementary projective geometry, let $Q$ be the unique quadric which is tangent to the plane $H_p$ at $p$ and contains the six points $\{ p_{ij}\mid 1\leq i < j \leq 4 \}$.  Furthermore, let $F_1\subset H_1$ be the unique quadric which is tangent to the line $L_{p_1}$ at $p$ and contains the three points $p_{12}$, $p_{13}$, $p_{1,4}$.  Then $Q\cap H_{p_1}=F_1=E_{p_1}$. Similarly, the ellipses $E_{p_2}$, $E_{p_3}$,  and $E_{p_4}$ are contained in the quadric $Q$.

\begin{center}
\includegraphics[width=3in]{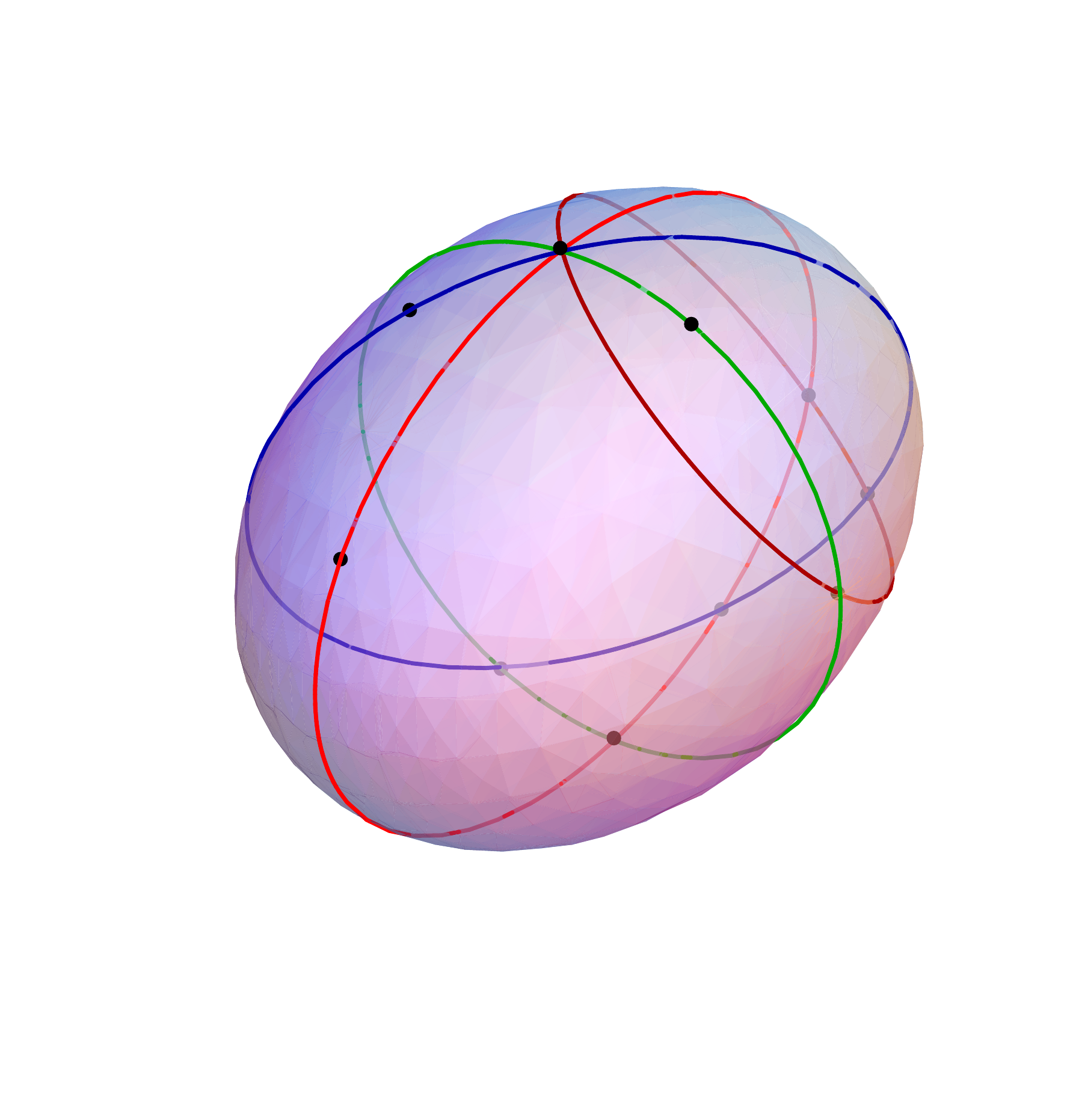}\rput(-3.7,5.85){$p$}\rput(-3.9,2.9){$p_{12}$}\rput(-2.3,3.3){$p_{13}$}\rput(-3.1,2.4){$p_{23}$}\rput(-.9,5.5){\textcolor{blue}{$E_1$}}\rput(-1.8,2.6){\textcolor[rgb]{0,.5,0}{$E_2$}}\rput(-2.4,6.75){\textcolor{red}{$E_3$}}\rput(-1.1,3.6){\textcolor[rgb]{.55,0,0.12}{$E_4$}}\rput(-4.6,5.45){$p_1$}\rput(-2.7,5.35){$p_2$}\rput(-4.95,3.8){$p_3$}
\end{center}

\medskip\noindent\textbf{Case b).} Three of planes  $H_{p_1}$, $H_{p_2}$, $H_{p_3}$, $H_{p_4}$  share a line. So without loss of generality, suppose $H_{p_1} \cap H_{p_2} \cap H_{p_3} =L_{12} = L_{13} =L_{23}$ and $p_{12} = p_{13} =p_{23}$. Let $\Gamma$ be the supporting plane of $K$ at $p_{12} = p_{13} =p_{23}$ and let $Q$ be the unique quadric tangent to $H_p$ at $p$ tangent to $\Gamma$ at  $p_{12} = p_{13} =p_{23}$  that contains arbitrarily chosen points $p_i\in E_i\setminus\{p, p_{12} = p_{13} =p_{23}\}$. Then $Q\cap H_1$ is the unique quadric contained in $H_1$, tangent to $L_{p_1}$ at $p$, tangent to $\Gamma \cap H_1$ at $p_{12} = p_{13} =p_{23}$ and that contains $p_1$.  This implies that the ellipse $E_1$ is contained in $Q$. Similarly, $E_2 \cup E_3 \subset Q$. Let $F_4$ be the unique quadric contained in $E_4$,  tangent to $L_{p_4}$ at $p$, and that contains the three distinct points $p_{14}$, $p_{24}$, and  $p_{34}$. Clearly $F_4=Q\cap H_4$ but also $F_4=E_4$, which implies that $E_4\subset Q$. 

\medskip\noindent\textbf{Case c).} The four planes $H_{p_1}$, $H_{p_2}$, $H_{p_3}$, $H_{p_4}$ share a line. Let $H_{p_1}\cap H_{p_2}\cap H_{p_3}\cap H_{p_4}=L_0$ and $L_0\cap \bd K=\{p,R\}$. At $R\in \bd K$ there is only one support plane $H_R$ of $K$. By using a projective homeomorphism if necessary, we may assume without loss of generality that $H_p$ is parallel to $H_R$ and also that $L_0$ is orthogonal to both $H_p$ and $H_R$. As in Case b), there is a quadric $Q$ containing $E_{p_1}$, $E_{p_2}$, $E_{p_3}$. By Lemma \ref{lem2}, we know that the sectional curvature at three different directions determines the sectional curvature at all other directions. In our situation, this implies that the curvature of the two ellipses $Q\cap H_{p_4}$ and $E_4$ are the same. This, together with the fact that both ellipses $Q\cap H_{p_4}$ and $E_4$ contained in $H_4$ are tangent at $p$ to $H_4\cap H_p$ and tangent at $R$ to $H_4\cap H_R$, implies that $Q\cap H_{p_4}=E_4$ and hence that $E_4\subset Q$.

We are ready to prove that there is a neighborhood $U$ of $\bd K$ at $p$ such that $U\subset Q$. Suppose this is not so; then there is a sequence $q_1, q_2,\dots, \in \bd K\setminus Q$ converging to $p$. For every $q_i\in \bd K$,  let $H_{q_i}$ be a plane through $q_i$ such that $H_{q_i}\cap \bd K$ is an ellipse $E_{q_i}$ of area greater than $\alpha$.  By considering a subsequence and renumbering if necessary, we may assume that the $H_{q_i}$ converge to a plane $H$ through $p$. The fact that $H_{q_i}\cap \bd K$ is an ellipse of area greater than $\alpha >0$ implies that $H$ is not a tangent plane of $K$. So let $L=H\cap H_p$.   We may assume without loss of generality that $L\neq L_{p_1}, L_{p_2},L_{p_3}$, so the plane $H$ intersects each of the ellipses $E_1,E_2,E_3$ at two points, one of them being $p$. Since $H_{q_i}$ converges to $H$, we may choose $n_0$ such that if $i>n_0$, then $H_{q_i}$ intersects the ellipse $E_j$ at two distinct points, $j=1,2,3$.  Therefore, since the quadrics $E_{q_i}$ and $Q\cap H_{q_i}$ share at least six points, they should be the same. This implies that $q_i\in Q$ for $i>n_0$, contradicting the fact that $q_1, q_2,\dots, \in \bd K\setminus Q$.

The fact that $\bd K$ is compact, connected and locally a quadric implies that $\bd K$ is an ellipsoid. \end{proof}

\section*{Acknowledgements}
The authors wish to thank Jes\'{u}s Jer\'{o}nimo and Edgardo Roldan-Pensado for many useful conversations regarding this paper. The second author wishes to acknowledge support from CONACyT under Project 166306 and support from PAPIIT--UNAM under Project IN112614.

\end{document}